\documentclass{amsart}
\usepackage{graphicx,amscd,color,amsmath,amsfonts,amssymb,geometry}
\usepackage[initials]{amsrefs}

\newtheorem{theorem}{Theorem}[section]
\newtheorem{proposition}[theorem]{Proposition}

\numberwithin{equation}{section}

\begin{document}

\title[On the real polynomial Bohnenblust--Hille inequality]{On the real polynomial Bohnenblust--Hille inequality}

\author[Campos \and Jim\'{e}nez \and Mu\~{n}oz \and Pellegrino \and Seoane]{J. R. Campos \and P. Jim\'{e}nez-Rodr\'{\i}guez \and G.A. Mu\~{n}oz-Fern\'{a}ndez \and D. Pellegrino \and J.B. Seoane-Sep\'{u}lveda}

\address{Departamento de Ci\^{e}ncias Exatas, \newline\indent Universidade Federal da Para\'{\i}ba, \newline\indent 58.297-000 - Rio Tinto, Brazil.}
\email{jamilson@dce.ufpb.br and jamilsonrc@gmail.com}

\address{Departamento de An\'{a}lisis Matem\'{a}tico,\newline\indent Facultad de Ciencias Matem\'{a}ticas, \newline\indent Plaza de Ciencias 3, \newline\indent Universidad Complutense de Madrid,\newline\indent Madrid, 28040, Spain.}
\email{pablo$\_$jimenez@mat.ucm.es}

\address{Departamento de An\'{a}lisis Matem\'{a}tico,\newline\indent Facultad de Ciencias Matem\'{a}ticas, \newline\indent Plaza de Ciencias 3, \newline\indent Universidad Complutense de Madrid,\newline\indent Madrid, 28040, Spain.}
\email{gustavo$\_$fernandez@mat.ucm.es}

\address{Departamento de Matem\'{a}tica, \newline\indent Universidade Federal da Para\'{\i}ba,
\newline\indent 58.051-900 - Jo\~{a}o Pessoa, Brazil.} \email{pellegrino@pq.cnpq.br and dmpellegrino@gmail.com}

\address{Departamento de An\'{a}lisis Matem\'{a}tico,\newline\indent Facultad de Ciencias Matem\'{a}ticas, \newline\indent Plaza de Ciencias 3, \newline\indent Universidad Complutense de Madrid,\newline\indent Madrid, 28040, Spain.}
\email{jseoane@mat.ucm.es}

\thanks{D. Pellegrino was supported by CNPq Grant 477124/2012-7 and INCT-Matem\'{a}tica.}


\maketitle

\begin{abstract}
It was recently proved by Bayart et al. that the complex polynomial Bohnenblust--Hille inequality is subexponential. We show that, for real scalars, this does no longer hold. Moreover, we show that, if $D_{\mathbb{R},m}$ stands for the real Bohnenblust--Hille constant for $m$-homogeneous polynomials, then $\displaystyle\lim\sup_{m}D_{\mathbb{R},m}^{1/m}=2.$
\end{abstract}

\section{Introduction}

If $E$ is a Banach space, real or complex, we say that $P$ is a homogeneous
polynomial on $E$ of degree $m\in{\mathbb{N}}$ if there exists an $m$-linear
form $L$ on $E^{m}$ such that $P(x)=L(x,\ldots,x)$ for all $x\in E$. It is
customary to denote by $\widehat{L}$ the restriction of $L$ to the diagonal of
$E^{m}$. An old and widely known algebraic result establishes that for every
homogeneous polynomial $P$ of degree $m$ on $E$ there exists a unique
symmetric $m$-linear form $L$ on $E^{m}$, called the polar of $P$, such that
$P=\widehat{L}$. We denote by ${\mathcal{P}}(^{m}E)$, ${\mathcal{L}}(^{m}E)$
and ${\mathcal{L}}^{s}(^{m}E)$ the spaces of continuous $m$-homogeneous
polynomials, continuous $m$-linear forms and continuous symmetric $m$-linear
forms on $E$ respectively. It is well known that homogeneous polynomials or
$m$-multilinear forms are continuous on $E$ if and only if they are bounded,
respectively, over the unit ball ${\mathsf{B}}_{E}$ of $E$ or ${\mathsf{B}%
}_{E}^{m}$. In that case
\begin{align*}
\|P\|:  &  =\sup\{|P(x)|:x\in{\mathsf{B}}_{E}\},\\
\|L\|:  &  =\sup\{|L(x_{1},\ldots,x_{m})|:x_{1},\ldots x_{m}\in{\mathsf{B}%
}_{E}\},
\end{align*}
define a norm in ${\mathcal{P}}(^{m}E)$ and ${\mathcal{L}}(^{m}E)$
respectively. If $P\in{\mathcal{P}}(^{m}E)$, we shall refer to $\|P\|$ as the
polynomial norm of $P$ in $E$. This norm is very difficult to compute in most
cases, for which reason it would be interesting to obtain reasonably good
estimates on it. The $\ell_{p}$ norm of the coefficients of a given polynomial
on ${\mathbb{K}}^{n}$ (${\mathbb{K}}={\mathbb{R}}$ or ${\mathbb{C}}$) has also
been widely used in mathematics and is much easier to handle. Observe that an
$m$-homogeneous polynomial in ${\mathbb{K}}^{n}$ can be written as
\[
P(x)={\sum\limits_{\left\vert \alpha\right\vert =m}}a_{\alpha}x^{\alpha},
\]
where $x=(x_{1},\ldots,x_{n})\in{\mathbb{K}}^{n}$, $\alpha=(\alpha_{1}%
,\ldots,\alpha_{n})\in({\mathbb{N}}\cup\{0\})^{n}$, $|\alpha|=\alpha
_{1}+\cdots+\alpha_{n}$ and $x^{\alpha}=x_{1}^{\alpha_{1}}\cdots x_{n}%
^{\alpha_{n}}$. Thus we define the $\ell_{p}$ norm of $P$, with $p\geq1$, as
%\[
%|P|_{p}=\left( \sum_{|\alpha|=m}|a_{\alpha}|^{p}\right) ^{\frac}{1}{p}.
%\]
If $E$ has finite dimension $n$, then the polynomial norm $\|\cdot\|$ and the
$\ell_{p}$ norm $|\cdot|_{p}$ ($p\geq1$) are equivalent, and therefore there
exist constants $k(m,n), K(m,n)>0$ such that
\begin{equation}
\label{equ:equiv}k(m,n)|P|_{p}\leq\|P\|\leq K(m,n)|P|_{p},
\end{equation}
for all $P\in{\mathcal{P}}(^{m}E)$. The latter inequalities may provide a good
estimate on $\|P\|$ as long as we know the exact value of the best possible
constants $k(m,n)$ and $K(m,n)$ appearing in \eqref{equ:equiv}.

The problem presented above is an extension of the the well known polynomial
Bohnenblust-Hille inequality (polynomial BH inequality for short). It was
proved in \cite{bh} that there exists a constant $D_{m}\geq1$ such that for
every $P\in{\mathcal{P}}(^{m}\ell_{\infty}^{n})$ we have
\begin{equation}
|P|_{\frac{2m}{m+1}}\leq D_{m}\Vert P\Vert. \label{equ:BH}%
\end{equation}
Observe that \eqref{equ:BH} coincides with the first inequality in
\eqref{equ:equiv} for $p=\frac{2m}{m+1}$ except for the fact that $D_{m}$ in
\eqref{equ:BH} can be chosen in such a way that it is independent from the
dimension $n$. Actually Bohnenblust and Hille showed that $\frac{2m}{m+1}$ is
optimal in \eqref{equ:BH} in the sense that for $p<\frac{2m}{m+1}$, any
constant $D$ fitting in the inequality
\[
|P|_{p}\leq D\Vert P\Vert,
\]
for all $P\in{\mathcal{P}}(^{m}\ell_{\infty}^{n})$ depends necessarily on $n$.

The polynomial and multilinear Bohnenblust--Hille inequalities were overlooked
for a long period (see \cite{Boas}) and were only rediscovered in the last
few years; now these inequalities can be seen as an extension of the successful theory of
absolutely summing operators (see \cite{DJT}) and have fundamental importance in different
fields of Mathematics and Physics, such as Operator Theory, Fourier and
Harmonic Analysis, Complex Analysis, Analytic Number Theory and Quantum
Information Theory (see \cite{bps, Boas, boas, annals2011, DD, diniz2, monta, Nu3, psseo, Que} and references therein).

The best constants in \eqref{equ:BH} may depend on whether we consider the
real or the complex version of $\ell_{\infty}^{n}$, which motivates the
following definition
\[
D_{{\mathbb{K}},m}:=\inf\left\{  D>0:\text{$|P|_{\frac{2m}{m+1}}\leq D\Vert
P\Vert$, for all $n\in{\mathbb{N}}$ and $P\in{\mathcal{P}}(^{m}\ell_{\infty
}^{n})$}\right\}  .
\]
If we restrict attention to ${\mathcal{P}}(^{m}\ell_{\infty}^{n})$ for some
$n\in{\mathbb{N}}$, then we define
\[
D_{{\mathbb{K}},m}(n):=\inf\left\{  D>0:\text{$|P|_{\frac{2m}{m+1}}\leq D\Vert
P\Vert$ for all $P\in$}{\mathcal{P}}(^{m}\ell_{\infty}^{n})\right\}  .
\]
Note that $D_{{\mathbb{K}},m}(n)\leq D_{{\mathbb{K}},m}$ for all
$n\in{\mathbb{N}}$.

It was recently shown in \cite{bps} that the complex polynomial
Bohnenblust--Hille inequality is, at most, subexponential, i.e., for any
$\varepsilon>0$, there is a constant $C_{\varepsilon}>0$ such that
$D_{{\mathbb{C}},m}\leq C_{\varepsilon}\left(  1+\varepsilon\right)  ^{m}$ for
all positive integers $m$. The main motivation of this paper are the
following problems:

\begin{itemize}

\item[(I)] Is the \emph{real} polynomial BH inequality subexponential?

\item[(II)] What is the optimal growth of the real polynomial BH inequality?

\end{itemize}

\noindent We provide the final answer to these previous problems by showing that $$\lim\sup_{m}D_{\mathbb{R},m}^{1/m}=2.$$

\section{The upper estimate\label{section3}}

The proof of the subexponentiality of the complex BH inequality given in \cite{bps} lies heavily in arguments restricted to complex scalars (it uses, for instance estimates from \cite{BAYART} for complex scalars); so a simple adaptation for the real case does not work. The calculation of the upper estimate of the BH inequality is quite simplified by the use of complexifications of polynomials. In particular we are interested in the
following deep result due to Visser \cite{Visser}, which generalizes and old result of Chebyshev:

\begin{theorem}
[Visser, \cite{Visser}, 1946]\label{theoremabe} Let
\[
P(y_{1}, \ldots, y_{n}) = \sum_{|\alpha|\le m} a_{\alpha} y_{1}^{\alpha_{1}}
\cdots y_{n}^{\alpha_{n}},
\]
with $\alpha= (\alpha_{1}, \ldots, \alpha_{n})$, $|\alpha| = \alpha_{1}+
\cdots+ \alpha_{n}$, be a polynomial of total degree at most $m \in\mathbb{N}
$ in the variables $y_{1},\ldots, y_{n}$ and with real coefficients
$a_{\alpha}$. Suppose $0\leq k\leq m$ and $P_{k}$ is the homogeneous
polynomial of degree $k$ defined by
\[
P_{k}(y_{1}, \ldots, y_{n}) = \sum_{|\alpha|= k} a_{\alpha} y_{1}^{\alpha_{1}}
\cdots y_{n}^{\alpha_{n}}.
\]
Then we have
\[
\max_{z_{1},\ldots,z_{n}\in{\mathbb{D}}} |P_{m}(z_{1}, \ldots, z_{n}%
)|\le2^{m-1} \cdot\max_{x_{1}, \ldots, x_{n} \in[-1,1]} |P(x_{1}, \ldots,
x_{n})|,
\]
where ${\mathbb{D}}$ stands for the closed unit disk in ${\mathbb{C}}$. In
particular, if $P$ is homogeneous, then
\[
\max_{z_{1},\ldots,z_{n}\in{\mathbb{D}}} |P(z_{1}, \ldots, z_{n})|\le2^{m-1}
\cdot\max_{x_{1}, \ldots, x_{n} \in[-1,1]} |P(x_{1}, \ldots, x_{n})|.
\]
Moreover, the constant $2^{m-1}$ cannot be replaced by any smaller one.
\end{theorem}

Let $P:\ell_{\infty}^{n}\left(  \mathbb{R}\right)  \rightarrow\mathbb{R}$ be
an $m$-homogeneous polynomial
\[
P(x)={\textstyle\sum\limits_{\left\vert \alpha\right\vert =m}}a_{\alpha
}x^{\alpha}%
\]
and consider the complexification $P_{\mathbb{C}}:\ell_{\infty}^{n}\left(
\mathbb{C}\right)  \rightarrow\mathbb{C}$ of $P$ given by
\[
P_{\mathbb{C}}(z)={\sum\limits_{\left\vert \alpha\right\vert =m}}a_{\alpha
}z^{\alpha}.
\]

From Theorem \ref{theoremabe} above we know that
\begin{equation}
\left\Vert P_{\mathbb{C}}\right\Vert \leq2^{m-1}\left\Vert P\right\Vert .
\label{pm}%
\end{equation}
Thus, since the complex polynomial Bohnenblust--Hille inequality is
subexponential, for all $\varepsilon>0$ there exists $C_{\varepsilon}>1$ such
that%
\begin{equation}
|P|_{\frac{2m}{m+1}}=|P_{\mathbb{C}}|_{\frac{2m}{m+1}}\leq C_{\varepsilon
}\left(  1+\varepsilon\right)  ^{m}\left\Vert P_{\mathbb{C}}\right\Vert
\label{889}%
\end{equation}
and combining (\ref{pm}) and (\ref{889}) we conclude that%
\[
\lim\sup_{m}D_{\mathbb{R},m}^{1/m}\leq2.
\]

\bigskip\bigskip As we mentioned earlier, Bayart et al. proved, recently, that the complex polynomial Bohnenblust--Hille inequality is subexponential (see \cite{bps}). The following result shows that the exponential growth of the real polynomial BH inequality is sharp in a very strong way: the exponential bound can not be reduced in any sense, i.e., there is an exponential lower bound for $D_{\mathbb{R},m}$ which holds for every $m \in \mathbb{N}$.

\begin{theorem}
\label{the:1.177}%

\[
D_{\mathbb{R},m}>\left(  \frac{2\sqrt[4]{3}}{\sqrt{5}}\right)  ^{m}>\left(
1.17\right)  ^{m}%
\]
for all positive integers $m>1$.
\end{theorem}

\begin{proof}
Let $m$ be an even integer. Consider the $m$-homogeneous polynomial
\[
R_{m}(x_{1},\ldots,x_{m})=\left(  x_{1}^{2}-x_{2}^{2}+x_{1}x_{2}\right)
\left(  x_{3}^{2}-x_{4}^{2}+x_{3}x_{4}\right)  \cdots\left(  x_{m-1}^{2}%
-x_{m}^{2}+x_{m-1}x_{m}\right)  .
\]
Since $\left\Vert R_{2}\right\Vert =5/4,$ it is simple to see that
\[
\left\Vert R_{m}\right\Vert =\left(  5/4\right)  ^{m/2}.
\]
From the BH inequality for $R_{m}$ we have
\[
\left(  {\sum\limits_{\left\vert \alpha\right\vert =m}}\left\vert a_{\alpha
}\right\vert ^{\frac{2m}{m+1}}\right)  ^{\frac{m+1}{2m}}\leq D_{\mathbb{R}%
,m}\left\Vert R_{m}\right\Vert ,
\]
that is,
\[
D_{\mathbb{R},m}\geq\frac{\left(  3^{\frac{m}{2}}\right)  ^{\frac{m+1}{2m}}%
}{\left(  \frac{5}{4}\right)  ^{\frac{m}{2}}}\geq\frac{\left(  \sqrt
{3}\right)  ^{\frac{m+1}{2}}}{\left(  \frac{5}{4}\right)  ^{\frac{m}{2}}%
}>\left(  \frac{2\sqrt[4]{3}}{\sqrt{5}}\right)  ^{m}.
\]
Now let us suppose that $m$ is odd. Keeping the previous notation, consider
the $m$ homogeneous polynomial%
\[
R_{m}\left(  x_{1},...,x_{2m}\right)  =\left(  x_{2m}+x_{2m-1}\right)
R_{m-1}\left(  x_{1},...,x_{m-1}\right)  +\left(  x_{2m}-x_{2m-1}\right)
R_{m-1}\left(  x_{m},...,x_{2m-2}\right)  .
\]
So we have%
\[
D_{\mathbb{R},m}\geq\frac{\left(  4\cdot3^{\frac{m-1}{2}}\right)  ^{\frac
{m+1}{2m}}}{2\cdot\left(  \frac{5}{4}\right)  ^{\frac{m-1}{2}}}>2^{m-1+\frac
{1}{m}}\left(  \frac{\sqrt[4]{3}}{\sqrt{5}}\right)  ^{m-1}>\left(
\frac{2\sqrt[4]{3}}{\sqrt{5}}\right)  ^{m-1}.
\]

\end{proof}

\bigskip

\section{The lower estimate\label{ii1}}

Using our previous results, in order to show that $\lim\sup_{m}D_{\mathbb{R}%
,m}^{1/m}=2$, we just need the following theorem:
%%%%%%%%%%%%%%%%%%%

\begin{theorem}
\label{the:Q_2k} If $k\in{\mathbb{N}}$ is fixed, then
\[
\lim\sup_{m}D_{\mathbb{R},m}^{1/m}(2^{k})\geq2^{1-2^{-k}}.
\]
Therefore,%
\[
\lim\sup_{m}D_{\mathbb{R},m}^{1/m}\geq2.
\]

\end{theorem}

\begin{proof}
Consider the sequence of polynomials (with norm $1$) defined recursively by
\begin{align*}
Q_{2}(x_{1},x_{2})  &  =x_{1}^{2}-x_{2}^{2},\\
Q_{2^{m}}(x_{1},\ldots,x_{2^{k}})  &  =Q_{2^{m-1}}(x_{1},\ldots,x_{2^{m-1}%
})^{2}-Q_{2^{m-1}}(x_{2^{m-1}+1},\ldots,x_{2^{m}})^{2}.
\end{align*}
Let us show (by induction on $m$) that
\begin{equation}
\label{eq_m}|Q_{2^{m}}^{n}|_{\infty} \ge\left(  \frac{2^{n}}{n+1} \right)
^{2^{m}-1}%
\end{equation}
for every natural number $n$. The case $m=1$ comes from the fact that, since
\[
\displaystyle 2^{n} = \sum_{k=0}^{n} \binom{n}{k} \le(n+1) \text{max}_{0\le k
\le n} \binom{n}{k},
\]
the $2n$-homogeneous polynomial $Q_{2}^{n}$ admits the following estimate:
\begin{equation}
\label{eq_uno}|Q_{2}^{n}|_{\frac{4n}{2n+1}} \ge|Q_{2}^{n}|_{\infty} =
\text{max}_{0\le k \le n} \binom{n}{k} \ge\frac{2^{n}}{n+1}.
\end{equation}
Let us now suppose that equation \eqref{eq_m} holds for some $m$, and notice
that
\begin{equation}
\label{hip_ind}Q_{2^{m}+1}^{n}(x_{1}, x_{2}, \ldots, x_{2^{m+1}}) = \sum
_{k=0}^{n} \binom{n}{k} (-1)^{n-k} Q_{2^{m}}^{2k} (x_{1}, \ldots, x_{2^{m}})
Q_{2^{m}}^{2(n-k)} (x_{2^{m} + 1}, \ldots, x_{2^{m+1}}).
\end{equation}
The coefficient of maximal absolute value in a product of polynomials in
disjoint sets of variables is the product of the respective maximal
coefficients, thus
\[
|Q_{2^{m}+1}^{n}|_{\infty} = \text{max}_{0 \le k \le n} \binom{n}{k}
|Q_{2^{m}}^{2k}|_{\infty} |Q_{2^{m}}^{2(n-k)}|_{\infty} \ge\text{max}_{0 \le k
\le n} \binom{n}{k} \left(  \frac{2^{2n}}{(2k+1)(2n-2k+1)}\right)  ^{2^{m} -
1}
\]
by the induction hypothesis. However, $(2k+1)(2n-2k+1) \le(n+1)^{2}$ when $0
\le k \le n$; thus
\[
|Q_{2^{m+1}}^{n}|_{\infty} \ge\left(  \frac{2^{n}}{n+1}\right)  ^{2^{m+1}-2}
\text{max}_{0 \le k \le n} \binom{n}{k} \ge\left(  \frac{2^{n}}{n+1}\right)
^{2^{m+1}-1},
\]
by equation \eqref{eq_uno}. Therefore, the formula given in \eqref{eq_m} holds
for every positive integer $m$.

\noindent Next, every $n$-homogeneous polynomial $P$ admits the clear estimate
given by
\[
|P|_{\frac{2n}{n+1}}\geq|P|_{\infty},
\]
from which equation \eqref{eq_m} yields that
\[
D_{{\mathbb{R}},n2^{m}}(2^{m})\geq\left(  \frac{2^{n}}{n+1}\right)  ^{2^{m}%
-1},
\]
and the proof follows straightforwardly.
\end{proof}

\section{Contractivity in finite dimensions: complex versus real scalars}\label{section5}

We remark that the complex polynomial Bohnenblust-Hille constants for
polynomials on ${\mathbb{C}}^{n}$, with $n\in{\mathbb{N}}$ fixed, are contractive.

\begin{proposition}
\label{o9} For all $n\geq2$ the complex polynomial BH inequality is
contractive in $\mathcal{P}(^{m}\ell_{\infty}^{n}).$ More precisely, for all
fixed $n\in{\mathbb{N}}$, there are constants $D_{m},$ with $\lim
_{m\rightarrow\infty}D_{m}=1$, so that
\[
|P|_{\frac{2m}{m+1}}\leq D_{m}\Vert P\Vert
\]
for all $P\in\mathcal{P}(^{m}\ell_{\infty}^{n}).$
\end{proposition}

\begin{proof}
Let $P(z)=\sum_{|\alpha|=m}c_{\alpha}z^{\alpha}$ and $f(t)=P(e^{it_{1}}%
,\ldots,e^{it_{n}})=\sum_{|\alpha|=m}c_{\alpha}e^{i\alpha t}$, where
$t=(t_{1},\ldots,t_{n})\in{\mathbb{R}}^{n}$ $\alpha\in({\mathbb{N}}%
\cup\{0\})^{n}$ and $\alpha t=\alpha_{1}t_{1}+\cdots+\alpha_{n}t_{n}$. .
Observe that if $\Vert f\Vert$ denotes the sup norm of $f$ on $[-\pi,\pi]$, by
the Maximum Modulus Principle $\Vert f\Vert=\Vert P\Vert$. Also, due to the
orthogonality of the system $\{e^{iks}:k\in{\mathbb{Z}}\}$ in $L^{2}([-\pi
,\pi])$ we have
\[
\Vert P\Vert^{2}=\Vert f\Vert^{2}\geq\frac{1}{2\pi}\int_{-\pi}^{\pi}%
|f(t)|^{2}dt=\sum_{|\alpha|=m}|c_{\alpha}|^{2}=|P|_{2}^{2},
\]
from which $|P|_{2}\leq\Vert P\Vert$. On the other hand it is well known that
in ${\mathbb{K}}^{d}$ we have
\begin{equation}
|\cdot|_{q}\leq|\cdot|_{p}\leq d^{\frac{1}{p}-\frac{1}{q}}|\cdot|_{q},
\label{equ:lpnorms}%
\end{equation}
for all $1\leq p\leq q$. Since the dimension of $\mathcal{P}(^{m}\ell_{\infty
}^{n})$ is ${\binom{{m+n-1}}{{n-1}}}$, the result follows from $|P|_{2}%
\leq\Vert P\Vert$ by setting in \eqref{equ:lpnorms} $p=\frac{2m}{m+1}$, $q=2$
and $d={\binom{{m+n-1}}{{n-1}}}$. So $D_{m}={{\binom{{m+n-1}}{{n-1}}}}%
^{\frac{1}{2m}}$ and since%
\[
\lim_{m\rightarrow\infty}{{\binom{{m+n-1}}{{n-1}}}}^{\frac{1}{2m}}=1 ,
\]
the proof is done.
\end{proof}

The next result shows that the \textit{real version} of Proposition \ref{o9} is
not valid; we stress that Theorem \ref{the:1.177} cannot be used here since it
uses polynomials in a growing number of variables.

\begin{theorem}
\label{o910} For all fixed positive integer $N\geq2$, the exponentiality
of the real polynomial Bohnenblust-Hille inequality in $\mathcal{P}(^{m}%
\ell_{\infty}^{N})$ cannot be improved. More precisely,
\[
\lim\sup_{m}D_{\mathbb{R},m}^{1/m}(N)\geq\sqrt[8]{27}\approx1.5098
\]
for all $N\geq2$.
\end{theorem}

\begin{proof}
If suffices to set $N=2$ and prove that, for $m=4n$,%
\[
D_{{\mathbb{R}},4n}\left(  2\right)  \geq\sqrt[4]{\frac{4}{m\pi}}\left(
\sqrt[8]{27}\right)  ^{m}.
\]
Consider the $4$-homogeneous polynomial given by
\[
P_{4}(x,y)=x^{3}y-xy^{3}=xy(x^{2}-y^{2}).
\]
A straightforward calculation shows that $P_{4}$ attains its norm at $\pm
(\pm\frac{1}{\sqrt{3}},1)$ and $\pm(1,\pm\frac{1}{\sqrt{3}})$ and that $\Vert
P_{4}\Vert=\frac{2\sqrt{3}}{9}$. On the other hand $\Vert P_{4}^{n}%
\Vert=\left(  \frac{2\sqrt{3}}{9}\right)  ^{n}$ and
\[
P_{4}(x,y)^{n}=x^{n}y^{n}\sum_{k=0}^{n}{\binom{{n}}{{k}}}(-1)^{k}%
x^{2k}y^{2(n-k)}.
\]
Hence, if $\mathbf{a}$ is the vector of the coefficients of $P_{4}$, using the
fact that $|\cdot|_{\frac{8n}{4n+1}}\geq|\cdot|_{2}$ (notice that here
$|\cdot|_{2}$ is the Euclidian norm), we have
\begin{align}
D_{{\mathbb{R}},4n}(2)  &  \geq\frac{|\mathbf{a}|_{\frac{8n}{4n+1}}}{\Vert
P_{4}\Vert^{n}}=\frac{\left[  \sum_{k=0}^{n}{\binom{{n}}{{k}}}^{\frac
{8n}{4n+1}}\right]  ^{\frac{4n+1}{8n}}}{\left(  {\frac{2\sqrt{3}}{9}}\right)
^{n}}\nonumber\\
&  \geq\frac{\left[  \sum_{k=0}^{n}{\binom{{n}}{{k}}}^{2}\right]  ^{\frac
{1}{2}}}{\left(  {\frac{2\sqrt{3}}{9}}\right)  ^{n}}=\frac{\sqrt{\binom{{2n}%
}{{n}}}}{\left(  {\frac{2\sqrt{3}}{9}}\right)  ^{n}}=\frac{\sqrt{(2n)!}%
}{\left(  {\frac{2\sqrt{3}}{9}}\right)  ^{n}n!}. \label{ali:lower_estimate}%
\end{align}
Above we have used the well known formula
\[
\sum_{k=0}^{n}{\binom{{n}}{{k}}}^{2}={\binom{{2n}}{{n}}}.
\]
Using Stirling's approximation formula
\[
n!\sim\sqrt{2\pi n}\left(  \frac{n}{e}\right)  ^{n}%
\]
in \eqref{ali:lower_estimate} we have, for $m=4n$,
\[
D_{{\mathbb{R}},m}(2)=D_{{\mathbb{R}},4n}(2)\geq\frac{\sqrt{(2n)!}}{\left(
{\frac{2\sqrt{3}}{9}}\right)  ^{n}n!}\sim\frac{\sqrt{2\sqrt{n\pi}\left(
\frac{2n}{e}\right)  ^{2n}}}{\left(  {\frac{2\sqrt{3}}{9}}\right)  ^{n}%
\sqrt{2\pi n}\left(  \frac{n}{e}\right)  ^{n}}=\sqrt[4]{\frac{4}{m\pi}}\left(
\sqrt[8]{27}\right)  ^{m}.
\]

\end{proof}

%%%%%%%%%%%%%%%%%%%

\end{document}